\documentclass[11pt]{article}
\usepackage[T1]{fontenc}
\usepackage{amsmath,amsthm,mathtools}
\usepackage{libertine}
\usepackage[libertine]{newtxmath}
\usepackage{microtype}
\usepackage[letterpaper,margin=1in]{geometry}
\usepackage[colorlinks=true,linkcolor=blue,citecolor=blue,urlcolor=blue]{hyperref}
\usepackage[capitalise,nameinlink,noabbrev]{cleveref}
\Crefformat{equation}{#2(#1)#3}
\Crefrangeformat{equation}{#3(#1)#4--#5(#2)#6}
\Crefmultiformat{equation}{#2(#1)#3}{ and #2(#1)#3}{, #2(#1)#3}{, and #2(#1)#3}

\newtheorem{theorem}{Theorem}
\newtheorem{lemma}{Lemma}[section]
\theoremstyle{remark}
\newtheorem*{remark}{Remark}
\newcommand{\F}{\mathcal F}
\newcommand{\Sfam}{\mathcal S}
\title{A Bound Below 2.8 for Tuza's Conjecture}
\author{Sichen Wang\\ \small Shenzhen MSU-BIT University\\ \small\texttt{wsc@smbu.edu.cn}}
\date{}
\hypersetup{pdftitle={A Bound Below 2.8 for Tuza's Conjecture},pdfauthor={Sichen Wang}}

\begin{document}
\maketitle
\begin{abstract}
Let $\nu(G)$ be the maximum number of edge-disjoint triangles in a graph $G$ and $\tau(G)$ the minimum number of edges meeting every triangle. Tuza conjectured that $\tau(G)\le2\nu(G)$. We prove that $\tau(G)\le(165/59)\,\nu(G)$. The constant $165/59\approx2.797$ improves the bound $66/23\approx2.870$ that Haxell proved in 1999. The key observation is that, for a suitable red-blue coloring, the families left over in Haxell's construction contain every triangle with exactly one red edge. Such a family $\F$ admits an exchange that forces certain red edges to lie in a single triangle once the blue edges of a maximum packing are deleted, which gives $\tau(\F)\le(8/3)\,\nu(\F)$.
\end{abstract}

\section{Introduction}

A triangle packing in a graph $G$ is a set of pairwise edge-disjoint triangles, and a triangle transversal is a set of edges meeting every triangle. Let $\nu(G)$ and $\tau(G)$ denote their maximum and minimum sizes. A transversal contains an edge of every packed triangle, and the edges of a maximum packing form a transversal, so $\nu(G)\le\tau(G)\le3\nu(G)$. Tuza~\cite{Tuza} conjectured that $\tau(G)\le2\nu(G)$ for every graph $G$. The constant $2$ would be best possible, as $K_4$ and $K_5$ show.

Haxell~\cite{Haxell} proved $\tau(G)\le(66/23)\,\nu(G)$, where $66/23\approx2.870$, and noted without proof that the argument can be pushed to $(1+\sqrt{481})/8\approx2.866$. Very recently, Yi~\cite{Yi} lowered the constant to $(162+4\sqrt3)/59\approx2.863$ by proving $\tau\le(1+\sqrt3)\nu$ for every family of triangles that admits a red-blue edge coloring with one red edge per triangle, and observed that this route alone cannot bring the constant below $54/19\approx2.842$. Independently of Yi's work, we prove the following bound, whose constant $165/59\approx2.797$ crosses this barrier and lies below $2.8$.

\begin{theorem}\label{thm:main}
Every finite simple graph $G$ satisfies
\[
 \tau(G)\le\frac{165}{59}\nu(G).
\]
\end{theorem}

Our starting point is the following exchange property. Let the edges of a graph be colored red and blue, and consider the family of all triangles with exactly one red edge. Fix a maximum packing in this family and delete the blue edges of its triangles. By maximality, every triangle of the family that survives meets the packing only in its red edge. Suppose now that a surviving triangle $uvy$ contains the red edge $uv$ of a packed triangle $uvx$, and that $xy$ is a red edge not used by the packing. Then $uvy$ is the only surviving triangle containing $uv$. Indeed, $uxy$ has exactly one red edge, and the family contains every such triangle. So if $uvz$ were a second surviving triangle containing $uv$, then replacing $uvx$ by $uxy$ and $uvz$ would enlarge the packing (\Cref{lem:unique}). A red edge that lies in a single surviving triangle can be omitted from any transversal containing the other two edges of that triangle. Two transversals of the family, built in \Cref{sec:colored}, allow this omission. Combining the two resulting bounds gives $\tau\le(8/3)\nu$ for every family of this kind, and every packed triangle whose red edge lies in no other triangle of the family lowers the bound by a further $2/3$ (\Cref{lem:colored}).

Families of this kind arise in Haxell's construction. The construction fixes a maximum packing $P$, calls the edges of its triangles old, and packs the remaining triangles in two further rounds: a maximum packing $A$ of triangles with one old edge, and, after the edges of $A$ are deleted, a packing $P'$ with as many triangles with two old edges as possible. Color new edges red and old edges blue. After the deletion of the edges of $A$, the triangles with two old edges form a family of the above kind, and so do the triangles that survive the further deletion of the old edges of $P'$. We build four transversals of $G$ from these packings, using books of triangles with a common edge, the bounds of \Cref{lem:colored}, and a random bipartition of the vertex set. Each gives an inequality between $\tau(G)$, $\nu(G)$, and the sizes of the packings, and a linear combination of the four inequalities yields $165/59$. \Cref{sec:colored} proves the exchange property and \Cref{lem:colored}, and \Cref{sec:general} carries out the construction. The complete proof has been formalized in Lean 4 with Mathlib.\footnote{\url{https://github.com/sichen-wang/tuza-lean}}

\section{Colored Triangle Families}\label{sec:colored}

Throughout, graphs are finite and simple, a triangle is identified with its set of three edges, and two triangles meet if they share an edge. For a family of triangles, a packing is a set of pairwise edge-disjoint members, a transversal is a set of edges meeting every member, and a packing is maximal if no member of the family can be added to it. Maximum packings are maximal. For a set $P$ of triangles, $E(P)$ is the union of their edge sets, and $H-X$ is the graph $H$ with the edge set $X$ deleted.

This section proves the three lemmas used in \Cref{sec:general}: the exchange property (\Cref{lem:unique}), a covering lemma for books (\Cref{lem:books}), and the two colored bounds (\Cref{lem:colored}). Let $H$ be a graph whose edges are colored red and blue, and let $\F(H)$ be the family of all triangles of $H$ with exactly one red edge. We write $\F$ for $\F(H)$ when $H$ is clear. For a packing $P$ in $\F(H)$, let $R(P)$ and $B(P)$ be the sets of red and blue edges of its members, and let
\[
 \F_P=\F(H-B(P))
\]
be the family of triangles with one red edge that survive the deletion of the blue edges of $P$. If $P$ is maximum, every member of $\F_P$ has its red edge in $R(P)$, since otherwise it could be added to $P$. Hence it meets $E(P)$ exactly in its red edge, and it meets exactly one member of $P$. We say that an edge is \emph{private} in a family of triangles if exactly one member of the family contains it.

\begin{lemma}\label{lem:unique}
Let $P$ be a maximum packing in $\F(H)$, let $uvx\in P$ have red edge $uv$, and let $uvy\in\F_P$. If $xy$ is a red edge outside $R(P)$, then $uv$ is private in $\F_P$.
\end{lemma}

\begin{proof}
Suppose that $uvz\in\F_P$ for some $z\ne y$. Since $uz\notin B(P)$ while $ux\in B(P)$, also $z\ne x$. The triangle $uxy$ has blue edges $ux,uy$ and red edge $xy$, so it belongs to $\F(H)$, and it meets $E(P)$ only in $ux$, because $uy\notin B(P)$ and $xy\notin R(P)$. The triangle $uvz$ meets $E(P)$ only in $uv$. As $z\notin\{x,y\}$, the triangles $uxy$ and $uvz$ are edge-disjoint, so replacing $uvx$ by both of them gives a packing in $\F(H)$ larger than $P$.
\end{proof}

The transversals below are built from books. A \emph{book} with \emph{spine} $uv$ is a set of triangles containing the edge $uv$. Its members are its \emph{pages}, and the union of their edge sets is its \emph{support}. If a book has exactly two pages $uvx$ and $uvy$ and $xy$ is an edge, then $xy$ is the \emph{opposite edge} of the book.

\begin{lemma}\label{lem:books}
Let $\mathcal B$ be a set of books with at most three pages each and pairwise edge-disjoint supports, and let $\mathcal T$ be a family of triangles such that every selection of one page from each book is a maximal packing in $\mathcal T$. Let $C$ consist of all edges of the one-page books and the spines of the other books. Then every member of $\mathcal T$ avoiding $C$ is $uxy$ or $vxy$ for some two-page book $\{uvx,uvy\}$ in $\mathcal B$. Consequently, $C$ together with the opposite edges of the two-page books is a transversal of $\mathcal T$ of size at most $\sum_{B\in\mathcal B}(4-|B|)$.
\end{lemma}

\begin{proof}
Let $T\in\mathcal T$ avoid $C$. If every book had a page edge-disjoint from $T$, choosing such a page in each book would give a maximal packing in $\mathcal T$ to which $T$ could still be added, a contradiction. So $T$ meets every page of some book $B$. Since all edges of a one-page book lie in $C$, $B$ has two or three pages, and $T$ avoids its spine $uv$, so $T$ contains $uz$ or $vz$ for every page $uvz$ of $B$. If $B$ had three pages, these would be three edges of $T$ with distinct ends outside $\{u,v\}$, but a triangle has only three vertices. So $B=\{uvx,uvy\}$, and $T$ contains one of $ux,vx$ and one of $uy,vy$. Two edges of a triangle share a vertex, so these are $ux,uy$ or $vx,vy$, and $T$ is $uxy$ or $vxy$. Finally, a book with $k$ pages contributes at most $4-k$ edges to the transversal.
\end{proof}

\begin{lemma}\label{lem:colored}
Let $P$ be a maximum packing in $\F(H)$ and $Q$ a maximum packing in $\F_P$, with $p=|P|$ and $q=|Q|$, and let $u$ be the number of members of $P$ whose red edge is private in $\F(H)$. Then
\begin{align}
 \tau(\F(H))&\le\tfrac83p-\tfrac23u,\label{eq:colored-bound}\\
 \tau(\F(H))&\le\tfrac94p+\tfrac32q-\tfrac14u.\label{eq:colored-refined}
\end{align}
\end{lemma}

\begin{proof}
Let $r$ be the number of members of $Q$ whose red edge is private in $\F_P$. We build two transversals of $\F$, which give
\[
 \tau(\F)\le3p-2q+r-u
 \qquad\text{and}\qquad
 \tau(\F)\le2p+3q-r,
\]
and then combine these bounds, once directly and once after applying \Cref{eq:colored-bound} to $\F_P$.

The first transversal comes from books. Each member of $Q$ shares its red edge with exactly one member of $P$, and the two form a book with that red edge as spine. Each unpaired member of $P$ forms a one-page book. No member of $P$ is paired twice, as $Q$ is a packing, and the supports are edge-disjoint, as the blue edges of $Q$ avoid $B(P)$. Every selection of one page per book is a packing in $\F$ of size $p=\nu(\F)$, hence maximal. A member of $P$ whose red edge is private in $\F$ is unpaired, since a page of $Q$ on the same spine would be a second member of $\F$ containing that edge.

Let $C$ consist of the $3(p-q)$ edges of the one-page books and the $q$ spines. By \Cref{lem:books}, a member of $\F$ avoiding $C$ is $uxy$ or $vxy$ for a two-page book $\{uvx,uvy\}$ with $uvx\in P$ and $uvy\in Q$. The two edges it shares with the pages are blue, so $xy$ is red, and $xy\notin R(P)$ since $R(P)\subseteq C$. By \Cref{lem:unique}, $uv$ is private in $\F_P$, so $uvy$ is one of the $r$ members of $Q$ with a private red edge. Adding the opposite edges of these at most $r$ books to $C$ therefore gives a transversal of $\F$. Finally, delete the red edges of the $u$ unpaired members of $P$ whose red edge is private in $\F$: the only member of $\F$ containing such an edge is that member of $P$, which is met by its blue edges in $C$. The transversal now has at most $3(p-q)+q+r-u$ edges, which is the first bound.

The second transversal starts from $B(P)\cup E(Q)$. A member of $\F$ avoiding $B(P)$ lies in $\F_P$ and, as $Q$ is maximal there, meets $E(Q)$. So $B(P)\cup E(Q)$ is a transversal of $\F$. Delete the red edges of the $r$ members of $Q$ whose red edge is private in $\F_P$: a member of $\F$ avoiding the remaining edges lies in $\F_P$ and contains a deleted edge, so it is the corresponding member of $Q$, which is met by its blue edges. The remaining $2p+3q-r$ edges give the second bound.

Twice the first bound plus the second gives $3\tau(\F)\le8p-q+r-2u\le8p-2u$, as $r\le q$, which is \Cref{eq:colored-bound}. Next, $Q$ is a maximum packing in $\F(H-B(P))=\F_P$, and $r$ of its members have a red edge that is private in $\F_P$, so \Cref{eq:colored-bound} for the graph $H-B(P)$ gives $\tau(\F_P)\le\tfrac83q-\tfrac23r$. Since $B(P)$ meets every member of $\F\setminus\F_P$,
\[
 \tau(\F)\le2p+\tfrac83q-\tfrac23r.
\]
Three times this bound plus the first bound gives $4\tau(\F)\le9p+6q-u-r\le9p+6q-u$, which is \Cref{eq:colored-refined}.
\end{proof}

\section{The General Bound}\label{sec:general}

\begin{proof}[Proof of \Cref{thm:main}]
We follow Haxell's construction~\cite{Haxell}: the packings below and the first transversal are taken from it, and the other three transversals rest on the lemmas of \Cref{sec:colored}. Choose a maximum triangle packing $P$ in $G$, and put $n=|P|=\nu(G)$ and $O=E(P)$. Call edges in $O$ old and all other edges new. A triangle has type $j$ if it contains exactly $j$ old edges. By maximality of $P$ there are no type-0 triangles.

Choose a maximum packing $A$ of type-1 triangles and put $a=|A|$. The first transversal comes from books formed by $P$ and $A$. Pair each member of $A$ with the member of $P$ containing its old edge. If two members of $A$ met the same member of $P$, replacing it by both would enlarge $P$, since a type-1 triangle meets $E(P)$ only in its old edge. These pairs and the remaining members of $P$ form $n$ books with $n+a$ pages. Their supports are edge-disjoint, since the new edges of $A$ avoid $O$, and every selection of one page per book is a packing of size $n=\nu(G)$, hence maximal. \Cref{lem:books} gives a transversal with at most $3(n-a)+2a$ edges, that is,
\begin{equation}
 \tau(G)\le3n-a.\label{eq:I1}
\end{equation}

Put $G'=G-E(A)$ and $O'=O\setminus E(A)$. Every type-1 triangle meets $E(A)$ by maximality of $A$, so every triangle of $G'$ has type $2$ or $3$, and $|O'|=3n-a$. Each remaining transversal of $G$ consists of $E(A)$ and a transversal of $G'$. Color the new edges of $G'$ red and its old edges blue. Then $\F(G')$ is the family $\F_2$ of type-2 triangles of $G'$.

Choose a packing $P'$ in $G'$ that maximizes first its number of type-2 members and then its size $m$. Let $B^*$ be the set of type-2 members of $P'$ and put $b=|B^*|$. Any packing in $\F_2$ is a packing in $G'$, so $B^*$ is a maximum packing in $\F_2$. The new edges of $P'$ are the new edges of the members of $B^*$, one for each. Let $N$ be their set, so that $R(B^*)=N$ and $|N|=b$. Since $A\cup P'$ is a packing in $G$, we have $a+m\le n$. The choice of $P'$ makes every packing in $G'$ with $m$ members, $b$ of them of type $2$, maximal in $G'$: adding a triangle would give a packing with at least $b$ type-2 members and $m+1$ members.

Let $B(P')$ be the set of old edges of $P'$, and put
\[
 \Sfam=\F(G'-B(P')),
\]
the family of type-2 triangles of $G'$ whose old edges avoid $E(P')$. Since $B(B^*)\subseteq B(P')$, we have $\Sfam\subseteq(\F_2)_{B^*}=\F(G'-B(B^*))$, and since $B^*$ is a maximum packing in $\F_2$, the red edge of every member of $\Sfam$ lies in $R(B^*)=N$. Choose a maximum packing $Q$ in $\Sfam$, put $c=|Q|$, and let $h$ be the number of members of $Q$ whose red edge is private in $\Sfam$.

The second transversal combines books formed by $B^*$ and $Q$ with a random bipartition. Pair each member of $Q$ with the member of $B^*$ containing its red edge. No member of $B^*$ is paired twice, as $Q$ is a packing. These pairs and the unpaired members of $B^*$ form $b$ books, whose supports are edge-disjoint since the old edges of $Q$ avoid $E(P')$. Every selection of one page per book consists of $b$ type-2 triangles and is therefore a maximum packing in $\F_2$. Let $W$ be the set of old edges of the one-page books, so $|W|=2(b-c)$. The set $C$ of \Cref{lem:books} for these books lies in $W\cup N$, since spines are edges of $N$ and a one-page book has its old edges in $W$ and its new edge in $N$.

Assign each vertex of $G'$ a bit, independently and uniformly at random, and call an edge crossing if its ends receive different bits. Select $W$ and every non-crossing edge of $O'$, then add the new edge of each triangle of $G'$ not yet met. A type-3 triangle has two vertices with the same bit and hence a non-crossing old edge, so the selected edges meet every triangle of $G'$. Their expected number is $|W|+\frac12(|O'|-|W|)$ plus the expected number of added new edges, which we bound in two parts.

A new edge is added only if both old edges of its triangle are crossing, and then it is non-crossing. Each edge of $N$ is non-crossing with probability $\frac12$, so the expected number of added edges in $N$ is at most $\frac b2$.

Now let an edge $xy\notin N$ be added as the new edge of a type-2 triangle $T$ not yet met. Then $T$ contains no edge of $N$ and avoids $W$, hence avoids $C$, and \Cref{lem:books} for these books and $\F_2$ shows that $T$ is $uxy$ or $vxy$ for a two-page book $\{uvx,uvy\}$ with $uvx\in B^*$ and $uvy\in Q$. The two edges that $T$ shares with the pages are old, so its new edge $xy$ is the opposite edge of the book. Since $uvy\in\Sfam\subseteq(\F_2)_{B^*}$ and $xy$ is a red edge outside $R(B^*)=N$, \Cref{lem:unique} shows that $uv$ is private in $(\F_2)_{B^*}$, hence in $\Sfam$, so $uvy$ is one of the $h$ members of $Q$ with a private red edge. Such a book contributes an added edge only if $uxy$ or $vxy$ is not yet met, which requires $x$ and $y$ to receive the same bit and $u$ or $v$ a different one. As $u,v,x,y$ are distinct, this has probability $\frac12-\frac18=\frac38$. Distinct added edges come from distinct books, as each book has one opposite edge, so the expected number of added edges outside $N$ is at most $\frac38h$. The expected total is at most
\[
 |W|+\frac{|O'|-|W|}{2}+\frac b2+\frac38h
 =\frac{|O'|}{2}+\frac32b-c+\frac38h,
\]
so some assignment of bits gives a transversal of $G'$ of at most this size. Adding $E(A)$ and using $|O'|=3n-a$ gives
\begin{equation}
 \tau(G)\le\tfrac32n+\tfrac52a+\tfrac32b-c+\tfrac38h.\label{eq:I2}
\end{equation}

For the third transversal, we add a third page to some of these books and include the remaining members of $P'$. Choose a maximum packing $Q_1$ in $\Sfam_Q=\F(G'-B(P')-B(Q))$ and put $d=|Q_1|$. The red edge of each member of $Q_1$ lies in $R(Q)$, and distinct members of $Q_1$ have distinct red edges. Add each member of $Q_1$ as a third page to the book whose spine is its red edge, and add the members of $P'\setminus B^*$ as one-page books. The old edges of $Q_1$ avoid $E(P')\cup E(Q)$, so the supports remain edge-disjoint. Every selection of one page per book is a packing in $G'$ with $m$ members, $b$ of them of type $2$, hence maximal in $G'$. These $m$ books have $m+c+d$ pages, so \Cref{lem:books} gives a transversal of $G'$ of size at most $3m-c-d$. Adding $E(A)$ and using $a+m\le n$ gives
\begin{equation}
 \tau(G)\le3n-c-d.\label{eq:I3}
\end{equation}

The fourth transversal consists of $E(A)$, the old edges of $P'$, and a transversal of $\Sfam$. The first two sets have at most $3a+3m-b\le3n-b$ edges. A triangle avoiding them lies in $G'-B(P')$ and is not of type $3$, since a type-3 triangle avoiding $B(P')$ would avoid $E(P')$ and could be added to $P'$. So it belongs to $\Sfam$. \Cref{lem:colored} applied to $G'-B(P')$, with $Q$, $Q_1$, $c$, $d$, $h$ in the roles of $P$, $Q$, $p$, $q$, $u$, gives $\tau(\Sfam)\le\frac94c+\frac32d-\frac14h$ by \Cref{eq:colored-refined}, so
\begin{equation}
 \tau(G)\le3n-b+\tfrac94c+\tfrac32d-\tfrac14h.\label{eq:I4}
\end{equation}

Multiply \Cref{eq:I1,eq:I2,eq:I3,eq:I4} by $20,8,19,12$, respectively, and add. The coefficients of $a,b,c,h$ cancel, giving
\[
 59\tau(G)\le165n-d\le165\nu(G).\qedhere
\]
\end{proof}

\begin{remark}
The weights are best possible for the system \Cref{eq:I1,eq:I2,eq:I3,eq:I4}. For $a=c=12n/59$, $b=m=39n/59$, and $d=h=0$, which satisfy the relations $a+m\le n$, $d\le c\le b\le m$, and $h\le c$ that hold in the proof, all four right-hand sides equal $165n/59$. So no nonnegative combination of the four inequalities gives a smaller constant.
\end{remark}

\end{document}